\documentclass[a4paper]{amsart}

\usepackage{amsmath,amssymb}
\usepackage[dvipsnames]{xcolor}

\newtheorem{theorem}{Theorem}[section]
\newtheorem{proposition}[theorem]{Proposition}
\newtheorem{corollary}[theorem]{Corollary}

\newtheorem{preremark}[theorem]{Remark}
\newtheorem{predefinition}[theorem]{Definition}
\newtheorem{preexample}[theorem]{Example}
\newtheorem{prenotation}[theorem]{Notation}
\newtheorem{preconjecture}[theorem]{Conjecture}

\newenvironment{remark}{\begin{preremark}\rm}{\end{preremark}}
\newenvironment{definition}{\begin{predefinition}\rm}
{\end{predefinition}}
\newenvironment{example}{\begin{preexample}\rm}{\end{preexample}}

\newenvironment{notation}{\begin{prenotation}\rm}{\end{prenotation}}

\def\OO{{\mathcal{O}}}
\def\QQ{\mathbb{Q}}
\def\X{{\mathbb{X}}}
\def\AA{\mathbb{A}}

\newcommand{\M}{{\mathfrak{M}}}
\newcommand{\N}{{\mathfrak{N}}}
\newcommand{\m}{{\mathfrak{m}}}

\let\epsilon=\varepsilon

\def\phi{{\varphi}}
\let\Psi=\varPsi
\let\Phi=\varPhi
\let\theta=\vartheta

\def\LT{\mathop{\rm LT}\nolimits}
\def\LC{\mathop{\rm LC}\nolimits}

\def\NF{\mathop{\rm NF}\nolimits}

\def\Hom{\mathop{\rm Hom}\nolimits}

\def\Spec{\mathop{\rm Spec}\nolimits}

\def\Cot{\mathop{\rm Cot}\nolimits}
\def\Tan{\mathop{\rm Tan}\nolimits}
\def\tail{\mathop{\rm tail}\nolimits}
\def\indets{\mathop{\rm indets}\nolimits}
\def\GFan{\mathop{\rm GFan}\nolimits}
\def\Ker{\mathop{\rm Ker}\nolimits}

\newcommand{\Lin}{\mathop{\rm Lin}\nolimits}
\newcommand{\edim}{\mathop{\rm edim}\nolimits}
\newcommand{\sepdim}{\mathop{\rm sepdim}\nolimits}
\newcommand{\LI}{\mathop{\rm LI}\nolimits}

\newcommand{\BO}{\mathbb{B}_{\mathcal{O}}}

\def\tfrac #1#2{{\textstyle\frac{#1}{#2}}}

\def\cocoa{\mbox{\rm
  C\kern-.13em o\kern-.07 em C\kern-.13em o\kern-.15em A}}
\def\apcocoa{\mbox{\rm
A\kern-0.13em p\kern -0.07em C\kern-.13em o\kern-.07 em C\kern-.13em
o\kern-.15em A}}

\begin{document}

\title{Cotangent Spaces and Separating Re-embeddings}

\author{Martin Kreuzer}
\address{Fakult\"at f\"ur Informatik und Mathematik, Universit\"at
Passau, D-94030 Passau, Germany}
\email{Martin.Kreuzer@uni-passau.de}

\author{Le Ngoc Long}
\address{Fakult\"at f\"ur Informatik und Mathematik, Universit\"at
Passau, D-94030 Passau, Germany and Department of Mathematics, 
University of Education, Hue University, 34 Le Loi Street, Hue City, Vietnam}
\email{lelong@hueuni.edu.vn}

\author{Lorenzo Robbiano}
\address{Dipartimento di Matematica, Universit\`a di Genova,
Via Dodecaneso 35,
I-16146 Genova, Italy}
\email{lorobbiano@gmail.com}

\begin{abstract}
Given an affine algebra $R=P/I$, where $P=K[x_1,\dots,x_n]$ is a polynomial 
ring over a field~$K$ and~$I$ is an ideal in~$P$, we study re-embeddings 
of the affine scheme $\Spec(R)$, i.e., presentations $R \cong P'/I'$
such that~$P'$ is a polynomial ring in fewer indeterminates.
To find such re-embeddings, we use polynomials~$f_i$ in the ideal~$I$ which are
coherently separating in the sense that they are of the form $f_i= z_i - g_i$ with an 
indeterminate~$z_i$ which divides neither a term in the support of~$g_i$
nor in the support of~$f_j$ for $j\ne i$. The possible numbers of such
sets of polynomials are shown to be governed by the Gr\"obner fan of~$I$.
The dimension of the cotangent space of~$R$ at a $K$-linear maximal ideal is
a lower bound for the embedding dimension, and if we find coherently separating
polynomials corresponding to this bound, we know that we have determined the
embedding dimension of~$R$ and found an optimal re-embedding. 
\end{abstract}



\maketitle

%
%

\section*{Introduction}

Given an affine scheme $\X$ embedded in an affine space $\mathbb{A}^n_K$ over a field~$K$,
it is a natural question to ask whether~$\X$ can be embedded into an affine space of
lower dimension. For instance, a classical result in algebraic geometry says that a smooth
affine variety of dimension~$d$ over an infinite field~$K$ can be embedded into~$\AA^{2d+1}_K$.
This was generalized by V.~Srinivas to the non-smooth case in~\cite{Sri}. The method to prove 
these results is to start with an embedding
$\X \longrightarrow \AA^n_K$ into a high-dimensional affine space and then to apply
general linear projections which re-embed~$\X$, i.e., which define isomorphisms on~$\X$.
Partial solutions to the above problem can be found in 
many papers. Among them, it is worth mentioning \cite{BC}, \cite{FR}, and \cite{LR}
where an approach of algebraic nature is used which is, to a certain extent, similar to 
our method.

Although the general problem of finding the minimal number of
algebra generators of an affine $K$-algebra seems to be very hard, it turns out to be possible
to re-embed some affine schemes in much lower dimensional affine spaces using linear projections
based on {\it separating indeterminates}. For instance, if a polynomial
in the given ideal is of the form $f= z-g$ where the indeterminate~$z$ does not divide any term
in the support of~$g$, one can eliminate~$z$ from the polynomials in the given ideal in an
obvious way. To eliminate several indeterminates at a time, one needs a {\it coherently separating}
set of polynomials in the sense defined below. The main advantage of such re-embeddings
is that they do not involve generic changes of coordinates. 
Here the possible dimensions of the ambient spaces after the re-embeddings are controlled by the
{\it Gr\"obner fan} of the given ideal which, alas, is hard to compute in general.

Can we reach the embedding dimension, i.e., the smallest dimension of the ambient space
in this way? In general, this is impossible, because the optimal re-embedding may not
be achieved via linear projections. One immediate obstruction is the
dimension of the tangent spaces at the $K$-rational points of the scheme. There can be no isomorphic
re-embedding into a space which has a lower dimension than one of the tangent spaces.
However, it is quite nice and useful to observe that if we reach this dimension at some point,
the resulting re-embedding is optimal and we have found the embedding dimension.
For instance, in the last section of the paper we show an example where our method allows us to 
compute an optimal embedding of a singular border basis scheme (see Example~\ref{ex:singBBS}).
More on good embeddings of border basis schemes will be discussed in a forthcoming paper.

To achieve these goals, we proceed as follows. 
In Section~\ref{Tangent and Cotangent Spaces} we consider  an affine algebra $R=P/I$, 
where $P=K[x_1,\dots,x_n]$ is a polynomial ring over a field~$K$ and~$I$ is an ideal in~$P$.
We observe that the computation of the tangent and cotangent
spaces of~$R$ at a maximal ideal $\m=\langle x_1-a_1,\dots,x_n-a_n\rangle /I$
corresponding to a $K$-rational point $(a_1,\dots,a_n)$ of~$\X = \Spec(R)$ boils down
to the calculation of the $\M$-linear part of~$I$, where $\M=\langle x_1-a_1,\dots,x_n-a_n\rangle
\supseteq I$ (see Prop.~\ref{prop:cotancomp}), and can be achieved efficiently using
any set of generators of~$I$ (see Prop.~\ref{prop:linearpart}).

In Section~\ref{Separating Re-embeddings}
we introduce the important concept of separating re-embeddings.
Given a set of distinct indeterminates $Z = \{ z_1,\dots,z_s\}$ in $X=\{x_1,\dots,x_n\}$,
we say that a set of polynomials $f_1,\dots,f_s\in I$ is {\it coherently $Z$-separating}
if each~$f_i$ is of the form $f_i = z_i -g_i$ with $g_i\in P$ such that~$z_i$ divides
neither a term in the support of~$g_i$ nor in the support of~$f_j$ with $j\ne i$.
This property is equivalent to the existence of Gr\"obner bases of~$I$ of a very special
shape (see~Prop.~\ref{prop:partofmin}) which in turn allow us to define an isomorphism
$\Phi:\; P/I \longrightarrow \widehat{P} / (I\cap \widehat{P})$, where
$\widehat{P} = K[X\setminus Z]$ has fewer indeterminates (see Thm.~\ref{thm:goodIso}). 
Clearly, the map~$\Phi$ corresponds to a re-embedding of $\X=\Spec(R)$ into a 
lower-dimensional affine space. We call it the {\it $Z$-separating re-embedding} of~$I$.

In Section~3 we study the possible choices of the sets of indeterminates~$Z$ for which
a coherently $Z$-separating set of polynomials exists in the given ideal~$I$.
Via the Gr\"obner bases mentioned above, they are related to the {\it Gr\"obner fan} of~$I$,
i.e., to the set of marked reduced Gr\"obner bases of~$I$. Here the sets~$Z$ correspond
to the sets of leading indeterminates of a reduced Gr\"obner basis of~$I$
(see Prop.~\ref{prop:equivZsepGB}). The existence of an {\it optimal separating re-embedding}
of~$I$ for which the ambient dimension equals the embedding dimension of~$R$
is therefore equivalent to the existence of a reduced Gr\"obner basis of~$I$
with enough leading terms which are indeterminates (see Prop.~\ref{prop:GFBest}).
This correspondence can be made 1-1 by introducing a suitable equivalence relation
on the set of reduced Gr\"obner bases of~$I$ (see Prop.~\ref{prop:equivalence}).

Consequently, we search for optimal re-embeddings in Section~4. The key observation
is that the dimension of the tangent space of~$R$ at any $K$-rational point is a
lower bound for the embedding dimension (see Thm.~\ref{thm:linandtang}) and that we
know we have reached the embedding dimension if we find {\it one} set of coherently
$Z$-separating polynomials of the correct cardinality (see~Cor.~\ref{cor:checkopt}).
Here the required cardinality of~$Z$ can be calculated efficiently using the technique
explained in Section~1. Suitable examples illustrate our main results.

The notation and definitions in this paper follow~\cite{KR1} and~\cite{KR2}. 
The calculations underlying most examples were performed using the computer algebra
system~\cite{CoCoA}.

\bigskip\bigbreak

\section{Tangent and Cotangent Spaces}
\label{Tangent and Cotangent Spaces}

In this section we introduce some concepts which play a fundamental role in the paper.
Most of them are well-known, however, for the sake of clarity and coherence of notation, 
we include them together with propositions, proofs, and examples.

Let $K$ be a field, let $P = K[x_1,\dots,x_n]$, let $p=(a_1,\dots,a_n)\in K^n$, and
let $\M$ be the linear maximal ideal $\M= \langle x_1-a_1,\dots,x_n -a_n\rangle $.
In the following we assume that~$I$ is an ideal in~$P$ such that $I \subseteq\M$.
Then we let $R=P/I$, we consider the maximal ideal $\m = \M/I$ of~$R$, and
we write $\X=\Spec(R)$ for the affine scheme associated to~$R$.
Notice that $I\subseteq \M$ implies $R/\m \cong P/\M \cong K$.

The first fundamental objects that we need in the following  are the
tangent space and the cotangent space of~$\X$ at a point~$p$. Let us recall them.

\begin{definition}\label{def:cotangent}
As above, let $p=(a_1,\dots,a_n)$, let $\M=\langle x_1-a_1, \dots, x_n-a_n\rangle$,
let~$I$ be an ideal in~$P$ which is contained in~$\M$, and let $\m=\M/I$.

\begin{enumerate}
\item[(a)] The $K$-vector space $\Cot_\m(R) = \m/\m^2$ is called the \textbf{Zariski cotangent space}
of~$R$ at the point~$p$. It is also denoted by $\Cot_p(\X)$.

\item[(b)] The dual vector space $\Tan_\m(R) = \Hom_K(\m/\m^2, K)$ of the cotangent space
is called the {\bf Zariski tangent space} of~$R$ at the point~$p$. It is also denoted by $\Tan_p(\X)$.
\end{enumerate}
\end{definition}

In order to calculate tangent and cotangent spaces, it will be useful to change the coordinate system
so that the point~$p$ under consideration is the origin. The following notation will come in handy.

\begin{notation}
In the above setting, we introduce new indeterminates $y_1,\dots,y_n$. Then we define the
$K$-algebra isomorphism 
$$
\phi_p:\; P \;\longrightarrow\;  K[y_1,\dots,y_n]\hbox{\qquad\rm given by\quad}x_i \mapsto y_i+a_i
\hbox{\quad\rm for\quad}i=1,\dots,n
$$
and its inverse
$$
\psi_p:\; K[y_1,\dots,y_n] \;\longrightarrow\;  P\hbox{\qquad\rm given by\quad}y_i \mapsto x_i-a_i
\hbox{\quad\rm for\quad}i=1,\dots,n.
$$
The homogeneous maximal ideal of $K[y_1,\dots,y_n]$ is $\N = \langle y_1,\dots,y_n\rangle$.
Clearly, we have $\N=\phi_p(\M)$.
For every $i\ge 0$, its homogeneous component of degree~$i$ is denoted by~$\N_i$.
In particular, we have $\N^i = \bigoplus_{j\ge i} \N_j$ for every $i\ge 1$.
\end{notation}

Now we are ready to introduce the following generalization of the linear part of a polynomial.

\begin{definition}
As above, let~$I$ be an ideal in~$P$ which is contained in the maximal ideal
$\M= \langle x_1-a_1, \dots, x_n-a_n \rangle$.
\begin{enumerate}
\item[(a)] Given $f\in P$, let $\ell\in \N_1$ be the homogeneous component of degree one 
of $\phi_p(f) = f(y_1+a_1,\dots,y_n+a_n)$.
Then the polynomial $\Lin_\M(f)= \psi_p(\ell) = \ell(x_1-a_1,\dots,x_n-a_n)$ is called
the {\bf $\M$-linear part} of~$f$.

\item[(b)] The $K$-vector space $\Lin_{\M}(I) = \langle \Lin_{\M}(f) \mid f\in I \rangle_K$ is called
the {\bf $\M$-linear part} of~$I$.
\end{enumerate}
\end{definition}

Notice that the result $\phi_p(f)$ of the substitution in part~(a) of this definition 
has no constant term if and only if $f\in\M$. Clearly, the definition yields an effective method 
for computing $\Lin_\M(f)$. The following example illustrates it.

\begin{example}\label{ex:calclin}
Let $P=\QQ[x_1,x_2]$, let $\M=\langle x_1-1,\; x_2-2\rangle $, and let $f = x_1^2 + x_2^2 - x_1 - 4$.
Then we calculate $g = f(y_1+1,\, y_2+2) = y_1^2 + y_2^2 +y_1 +4y_2$ and $\Lin_\N(g)= y_1+4y_2$.
Hence we obtain $\Lin_\M(f) = (x_1-1) + 4(x_2-2) = x_1 +4x_2 -9$. 
\end{example}

Now the cotangent and the tangent space of $\X=\Spec(R)$ at~$p$ can be described as follows.

\begin{proposition}\label{prop:cotancomp}
As in the above setting, let $p=(a_1,\dots,a_n)\in K^n$, 
let $\M=\langle x_1-a_1,\dots, x_n-a_n\rangle$, and 
let~$I$ be an ideal in~$P$ which is contained in~$\M$.
\begin{enumerate}
\item[(a)] Let $\M_1= \langle x_1-a_1, \dots, x_n-a_n \rangle_K$.
Then we have $\Lin_\M(I) = \M_1 \cap (I+\M^2)$.

\item[(b)] There is an isomorphism of $K$-vector spaces 
$\Cot_\m(R) \cong \M_1 / \Lin_\M(I)$.

\item[(c)] There is an isomorphism of $K$-vector spaces 
$$
\Tan_\m(R)\cong \{ v\in K^n \mid \ell(p+v)=0 \hbox{\ \it for all\ }\ell\in \Lin_\M(I)\}
$$
\end{enumerate}
\end{proposition}

\begin{proof} 
To prove~(a), we let $J=\phi_p(I)$.
By the definition of the linear part of an ideal,
we have to show $\Lin_\N(J) = \N_1 \cap (J+\N^2)$,
where we know that $J\subseteq \N$. Since the inclusion $\subseteq$
is trivially true, it suffices to prove $\supseteq$. Let $\ell\in \N_1$ be of the form
$\ell=f+g$ with $f\in J$ and $g\in \N^2$. Then the equality $f= \ell -g$ shows that
$\ell=\Lin_\N(f)$, and hence $\ell\in\Lin_\N(J)$, as claimed.

To prove~(b), we note that $\Cot_\m(R) = \m / \m^2 \cong
\M / (I+\M^2)$. Using $\N^i=\bigoplus_{j\ge i} \N_j$ for $i\ge 1$,
the isomorphism~$\psi_p$ yields direct sum decompositions
$\M=\bigoplus_{i\ge 1} \M_i$ and $\M^2 = \bigoplus_{i\ge 2} \M_i$,
where $\M_i = \psi_p(\N_i)$ for $i\ge 1$. Hence we get that $\Cot_\m(R)$
is the residue class vector space of~$\M_1$ modulo $\M_1 \cap (I+\M^2)$,
and the latter intersection equals $\Lin_\M(I)$ by~(a).

Claim~(c) is a consequence of~(b) and the observation that a linear form 
$\ell = c_1 (x_1-a_1)^\ast + \cdots + c_n (x_n-a_n)^\ast$ in the dual vector
space of~$\M_1$ is exactly the evaluation at $p+(c_1,\dots,c_n)\in K^n$.
\end{proof}

In view of this proposition, the tasks to compute the Zariski cotangent and tangent
spaces of~$R$ at~$p$ are reduced to computing the $\M$-linear part of~$I$.
This can be done as follows.

\begin{proposition}\label{prop:linearpart}
As in the above setting, let $p=(a_1,\dots,a_n)\in K^n$, 
let $\M=\langle x_1-a_1,\dots, x_n-a_n\rangle$, let
$f_1,\dots,f_s\in\M$, and let $I=\langle f_1,\dots,f_s\rangle$.
Then we have
$$
\Lin_\M(I) \;=\; \langle \Lin_\M(f_1), \dots, \Lin_\M(f_s) \rangle_K 
$$
\end{proposition}

\begin{proof}
For $i=1,\dots,s$, we write $f_i = \ell_i + g_i$ with $\ell_i=\Lin_\M(f_i)$
and $g_i \in \M^2$. Then we have $I+\M^2 = \langle \ell_1,\dots,
\ell_s\rangle + \M^2$ and therefore,  using
Proposition~\ref{prop:cotancomp}.a, we get
$\Lin_\M(I) = \M_1 \cap (I+\M^2) = 
\M_1 \cap (\langle \ell_1,\dots,\ell_s\rangle + \M^2)$. 
When we apply the isomorphism~$\phi_p$, we
see that the only homogeneous polynomials of degree one in 
the image of the ideal $\langle \ell_1,\dots,\ell_s\rangle + \M^2$
are the ones in the image of the vector space $\langle \ell_1, \dots, \ell_s\rangle_K$.
Thus we get $\Lin_\M(I)=\langle \ell_1, \dots, \ell_s\rangle_K$, as claimed.
\end{proof}

Notice that this proposition implies that Zariski tangent and 
cotangent spaces can be computed efficiently using any system of generators of~$I$.
If the ideal~$I$ is not contained in~$\M$, the conclusions of the proposition may fail,
as the following example shows.

\begin{example}\label{rex:nottruelinear}
Let $P=\QQ[x,y,z]$, let $\M=\langle x,y,z\rangle$, and let $I = \langle f_1, f_2\rangle$, 
where $f_1=x+1$ and $f_2 = y -z - z^2$.
Then $g = y f_1 -x f_2 = xz^2 +xz +y \in I$ implies that $y=\Lin_\M(g)$ is in the 
$\M$-linear part of $I$. However, $y$ is not contained in the ideal 
$\langle \Lin_\M(f_1),\; \Lin_\M(f_2)\rangle = \langle x,\; y-z \rangle$.
\end{example}

Let us use the above results to compute some cotangent and tangent spaces.

\begin{example}\label{ex:cotanglinear}
Let $P=\QQ[x,y,z]$, let $\M=\langle x,y,z\rangle$, and let $I=\langle f_1,f_2,f_3\rangle$,
where  $f_1= x^3 -x -z$, $f_2=y^2 +x$, and $f_3 = xy + z$. 
Then we get $\Lin_\M(I) = \langle \Lin_\M(f_1),\; \Lin_\M(f_2),\; \Lin_\M(f_3) \rangle =
\langle -x-z,\, x,\, z \rangle = \langle x,\, z\rangle$.
Hence the cotangent space of $R=P/I$ at $p=(0,0,0)$ is $\Cot_\m(R)=\m/\m^2 \cong  \langle x,y,z\rangle
/ \langle x,z,y^2\rangle \cong K\cdot \bar{y}$ and the tangent space is 
$\Tan_\m(R)= K \cdot (0,1,0)$. 
\end{example}

\begin{example}
Let $P=\QQ[x,y,z]$, let $\M=\langle x,y,z\rangle$, and let $I=\langle f_1,f_2\rangle$, where
$f_1 = x^3z -x$ and $f_2 = yz^2 - x$.
Then we have $\Lin_\M(I)= \langle x\rangle$. Therefore the cotangent space of~$R=P/I$
at the origin is the 2-dimensional vector space $\Cot_\m(R) = \m /\m^2 \cong \langle x,y,z\rangle /
\langle x,y^2,yz,z^2\rangle \cong K\cdot \bar{y} \oplus K\cdot \bar{z}$, and the tangent space is 
$\Tan_\m(R)= K\cdot (0,1,0) \oplus K\cdot (0,0,1)$.
\end{example}

\bigskip\bigbreak
%
%

\section{Separating Re-embeddings}
\label{Separating Re-embeddings}

In this section we continue to study affine $K$-algebras $R=P/I$,
where $P=K[x_1,\dots,x_n]$ is a polynomial ring over a field~$K$ and~$I$ is an ideal
in~$P$. The presentation $R=P/I$ corresponds to an embedding of the affine scheme
$\X=\Spec(R)$ into the affine $n$-space $\mathbb{A}^n_K$. Computing another presentation
$R \cong P'/I'$ with $P'=K[y_1,\dots,y_m]$ and an ideal~$I'$ in~$P'$ amounts to embedding~$\X$
into another affine space, preferably of lower dimension $m<n$.

Starting from this point, we assume that the linear maximal ideal under
consideration is $\M=\langle x_1,\dots,x_n\rangle$. We leave it to the interested reader
to generalize the subsequent results to arbitrary linear maximal ideals of~$P$ with the help
of Proposition~\ref{prop:linearpart}. Our goal is to construct alternative embeddings 
of~$\X$ using particular generators of the ideal~$I$.  
The main ideas and results of this section can be viewed as part of 
elimination theory which is explained in great detail for instance in Section 3.4 of~\cite{KR1}
(see Remark~\ref{rem:elimTO}).
The next definition provides a key construction.

\begin{definition}\label{def:separating}
In the above setting, let $\M=\langle x_1,\dots, x_n\rangle$, and let $f\in \M$. 
\begin{enumerate}
\item[(a)] The set of all indeterminates in $\{x_1,\dots,x_n\}$
which divide at least one of the terms in the support of~$f$ is called the
{\bf set of indeterminates} of~$f$ and is denoted by $\indets(f)$.

\item[(b)] Assuming $\Lin_\M(f)\ne 0$, let $z\in \indets(\Lin_\M(f))$, 
and let $c\in K\setminus \{0\}$ be the coefficient of~$z$ in~$f$. 
Then the polynomial $z - \frac{1}{c}\, f$ is called the
{\bf $z$-tail} of~$f$ and is denoted by $\tail_z(f)$.

\item[(c)] For every indeterminate $z\in \indets(\Lin_\M(f))$ such that
$z \notin \indets(\tail_z(f))$ we say that~$f$ is {\bf $z$-separating}.
\end{enumerate}
\end{definition}

The notion of being $z$-separating can be characterized as follows.

\begin{remark}\label{rem:sep=GB}
Let $f\in \M$ with $\Lin_\M(f)\ne 0$, and let $z\in \indets(\Lin_\M(f))$.
Then~$f$ is $z$-separating if and only if there exists a term ordering~$\sigma$
such that $z=\LT_\sigma(f)$. In this case the reduced $\sigma$-Gr\"obner basis of the
principal ideal $\langle f\rangle$ is $\{ \frac{1}{c}\, f \}$. 

Note that $\frac{1}{c}\, f = z - \tail_z(f)$ and that $\tail_z(f) \in \widehat{P}$,
where $\widehat{P}$ is the polynomial ring in the indeterminates $\{x_1,\dots,x_n\} \setminus \{z\}$
over~$K$. Consequently, the substitution $z \mapsto \tail_z(f)$ yields a $K$-algebra
isomorphism $P/\langle f\rangle \cong \widehat{P}$ which allows us to identify the
hypersurface $\Spec(P/\langle f\rangle )$ in~$\mathbb{A}^n_K$ with the affine 
space $\mathbb{A}^{n-1}_K$.
\end{remark}

The following example indicates a way how to use a separating polynomial
in order to find embeddings into affine spaces of lower dimension.

\begin{example}\label{ex:separating}
Let $P=\QQ[x,y,z]$, and let $I=\langle f,g \rangle$, where $f=y^3z -z^4 +2x$ 
and $g = z^2 -xy -y$. Clearly, the polynomial~$f$ is $x$-separating.
Let $\sigma$ be the elimination ordering ${\tt Elim}(x)$ (see~\cite{KR1}, Def.~1.4.10).
Then the reduced $\sigma$-Gr\"obner basis of~$I$ is
$$
\{ \; \tfrac{1}{2}\,f, \; y^4z -yz^4 +2z^2 -2y \} = 
\{\; x - (-\tfrac{1}{2}\, y^3z + \tfrac{1}{2}\, z^4),\;
y^4z -yz^4 +2z^2 -2y \}
$$
Hence~\cite{KR1}, Thm.~3.4.5, yields $I \cap \QQ[y,z] = \langle y^4z -yz^4 +2z^2 -2y\rangle$,
and the substitution $x\mapsto -\tfrac{1}{2}\, y^3z + \tfrac{1}{2}\, z^4$
induces a $\QQ$-algebra isomorphism 
$$
P/I \;\cong\; \QQ[y,z]/ \langle y^4z -yz^4 +2z^2 -2y\rangle
$$
Geometrically, we have embedded the subscheme $\Spec(P/I)$ of $\AA_\QQ^3$
into~$\AA_\QQ^2$.
\end{example}

This example motivates us to try to eliminate several indeterminates simultaneously.
To keep the notation manageable, we introduce the following setting.

\begin{notation}\label{not:YandZ}
The set of indeterminates of~$P$ is denoted by $X = \{x_1,\dots,x_n\}$.
Let $s\ge 1$. For $i=1,\dots,s$, we choose pairwise distinct indeterminates 
$z_i\in X$, and we let~$Z$ be the set $Z = \{ z_1,\dots,z_s \}$. 
Moreover, the remaining indeterminates $Y=X\setminus Z$
will be written as $Y = \{y_1,\dots,y_{n-s}\}$,
and we denote the polynomial ring in these indeterminates by $\widehat{P} = K[y_1,\dots,y_{n-s}]$.
Furthermore, given a term ordering~$\sigma$ on~$P$, its restriction to
${\widehat{P}}^{\mathstrut}$ is denoted by $\hat{\sigma}$.
\end{notation}

The following definition provides a condition for the simultaneous elimination of 
several indeterminates to work.

\begin{definition}\label{def:consistsep}
Let $s\ge 1$, let $f_1,\dots,f_s \in \M\setminus\{0\}$, and let 
$Z= \{z_1,\dots,z_s\}$ be a set of~$s$ distinct indeterminates in~$X$.
Then we say that the tuple $(f_1,\dots,f_s)$ is {\bf coherently $Z$-separating}
if the following two conditions are satisfied for every $i\in \{1,\dots,s\}$.
\begin{enumerate}
\item[(1)] The polynomial $f_i$ is $z_i$-separating.

\item[(2)] For $j\ne i$, we have $z_i \notin \indets(f_j)$.
\end{enumerate}
\end{definition}

Generalizing Remark~\ref{rem:sep=GB}, we can characterize coherently $Z$-separating tuples
of polynomials as follows.

\begin{proposition}\label{prop:partofmin}

Let $s\ge 1$, let $f_1,\dots,f_s \in \M\setminus\{0\}$, and let $Z = \{z_1,\dots,z_s\}$ be a set
of~$s$ distinct indeterminates in~$X$. Then the following conditions are equivalent.
\begin{enumerate}
\item[(a)] The tuple $(f_1,\dots,f_s)$ is coherently $Z$-separating.

\item[(b)] For every term ordering~$\sigma$ such that $z_i=\LT_\sigma(f_i)$ for $i=1,\dots,s$,
the reduced $\sigma$-Gr\"obner basis of $\langle f_1,\dots,f_s\rangle$ 
is of the form $\{ \frac{1}{c_1}\, f_1,\dots, \frac{1}{c_s}\,f_s \}$, where $c_i = \LC_\sigma(f_i)$ for $i=1,\dots,s$.

\item[(c)] For every proper ideal~$I$ in~$P$ containing $f_1,\dots,f_s$ and for every
term ordering~$\sigma$ such that $z_i=\LT_\sigma(f_i)$ for $i=1,\dots,s$, the ideal~$I$ 
has a $\sigma$-Gr\"obner basis of the form
$\{ \frac{1}{c_1}\, f_1,\dots, \frac{1}{c_s}\,f_s,\, g_1,\dots,g_t \}$, where $c_i = \LC_\sigma(f_i)$ 
for $i=1,\dots,s$, where the reduced $\sigma$-Gr\"obner basis of $\langle f_1,\dots,f_s\rangle$ 
is $\{ \frac{1}{c_1}\, f_1,\dots, \frac{1}{c_s}\,f_s \}$,
and where $\{g_1,\dots,g_t \}$ is the reduced $\hat{\sigma}$-Gr\"obner basis 
of $I\cap \widehat{P}$.
\end{enumerate}

\end{proposition}

\begin{proof}
First we show that (a) implies~(b). Let~$\sigma$ be a term ordering such that $\LT_\sigma(f_i)=z_i$
for $i=1,\dots,s$. Then the hypothesis implies that~$z_i$ does not divide any term in
the support of $\tail_{z_i}(f_i)$ or in the support of one of the polynomials $f_j$ with $j\ne i$.
Since the leading terms $z_i=\LT_\sigma(f_i)$ are pairwise coprime, the set 
$\{ \frac{1}{c_1}\, f_1,\dots, \frac{1}{c_s}\,f_s \}$ is a $\sigma$-Gr\"obner basis 
of $\langle f_1, \dots, f_s \rangle$. 
Using $z_i \notin \indets(f_j)$ for $i\ne j$, we see that this set is fully interreduced, 
and therefore it is the reduced $\sigma$-Gr\"obner basis  of $\langle f_1, \dots, f_s \rangle$.

Next we show that~(b) implies~(c).  Using $z_i=\LT_\sigma(f_i)$ for $i=1,\dots,s$ and the hypothesis
that~$I$ is not the unit ideal in~$P$, it follows that the indeterminates $z_1, \dots, z_s$ are part of a 
minimal system of generators of $\LT_\sigma(I)$. Hence the set 
$\{ \frac{1}{c_1}\, f_1,\dots, \frac{1}{c_s}\,f_s \}$ is part of a 
minimal $\sigma$-Gr\"obner basis of~$I$. Let $\{\tilde{g}_1,\dots,\tilde{g}_t \}$ be the
remaining elements of this minimal Gr\"obner basis. As the leading terms $\LT_\sigma(\tilde{g}_i)$
are not divisible by any indeterminate in~$Z$, we can replace $\tilde{g}_i$ by their normal
remainders $\hat{g}_i$ with respect to division by $(f_1,\dots,f_s)$. Then we have $\hat{g}_i
\in \widehat{P}$ for $i=1,\dots,t$. Finally, we can divide the polynomials $\hat{g}_i$ by their
leading coefficients and interreduce them to get the reduced $\hat{\sigma}$-Gr\"obner basis
$\{g_1,\dots,g_t\}$ of $I \cap \widehat{P}$.

The implication  (c)$\Rightarrow$(b) is clear, and the implication (b)$\Rightarrow$(a) follows directly
from the definition of a reduced Gr\"obner basis. Thus the proof is complete.
\end{proof}

The following remark says that suitable term orderings always exist.

\begin{remark}\label{rem:elimTO}
Notice that, for a coherently $Z$-separating tuple $(f_1,\dots,f_s)$, 
any elimination ordering~$\sigma$ for~$Z$ has the property $\LT_\sigma(f_i)=z_i$
required by conditions~(b) and~(c) of the proposition. 
\end{remark}

The next example illustrates the proposition.

\begin{example}\label{ex:cohersep}
Let $P = \QQ[x,y,z]$, and let $I=\langle f_1,f_2\rangle$, where 
$f_1 = x^2 -x -y$ and $f_2 = y^2 -z$.
\begin{enumerate}
\item[(a)] The polynomial~$f_1$ is $y$-separating and~$f_2$ is $z$-separating.
However, the tuple $(f_1, f_2)$ is not coherently $Z$-separating 
with respect to $Z=\{y, z\}$, since~$y$ appears in the $z$-tail of~$f_2$.

\item[(b)] Let $\sigma$ be the term ordering ${\tt Elim}(y,z)$.
Then the reduced $\sigma$-Gr\"obner basis of~$I$ is $\{-f_1, g\}$,
where $g = z -(x^2 -x )^2$. Clearly, the tuple $(f_1, g)$ is coherently $Z$-separating
with respect to $Z=\{y, z\}$.
\end{enumerate}
\end{example}

Proposition~\ref{prop:partofmin} suggests the following definition.

\begin{definition}\label{def:sepGB}
Let $s\ge 1$, let $Z = \{z_1,\dots,z_s\}$ be a set of~$s$ distinct indeterminates in~$X$,
and let $I\subseteq \M$ be an ideal.
\begin{enumerate}
\item[(a)] A term ordering~$\sigma$ on~$\mathbb{T}^n$ is called a {\bf $Z$-separating term ordering} 
for~$I$ if there exist polynomials $f_1,\dots,f_s \in I \setminus \{0\}$ such that
$z_i=\LT_\sigma(f_i)$ for $i=1,\dots,s$ and $(f_1,\dots,f_s)$ is coherently $Z$-separating.

\item[(b)] Let~$\sigma$ be a $Z$-separating term ordering for~$I$ on~$\mathbb{T}^n$.
Then a $\sigma$-Gr\"obner basis of~$I$ is called a {\bf $Z$-separating $\sigma$-Gr\"obner basis} 
of~$I$ if it is of the form $\{ \frac{1}{c_1}\, f_1,\dots, \frac{1}{c_s}\,f_s, g_1,\dots,g_t \}$
where $(f_1,\dots,f_s)$ is a tuple of coherently $Z$-separating polynomials, 
where we have $c_i = \LC_\sigma(f_i)$ and $z_i=\LT_\sigma(f_i)$ for $i=1,\dots,s$, and where 
$\{g_1,\dots,g_t \}$ is the reduced $\hat{\sigma}$-Gr\"obner basis 
of $I\cap \widehat{P}$.
\end{enumerate}
\end{definition}

In polynomial system solving, the term ordering {\tt Lex} frequently acts as a separating term
ordering, as in the following example.

\begin{example}
Let $K$ be a perfect field, let $P=K[x_1,\dots,x_n]$, and let $I\subseteq \M$ 
be a 0-dimensional radical ideal in normal $x_n$-position. Then the Shape Lemma (see~\cite{KR1},
Thm.~3.7.25) says that $\sigma={\tt Lex}$ is a $Z$-separating term ordering for~$I$ when 
$Z = \{x_1,\dots,x_{n-1}\}$.
\end{example}

Notice that a $Z$-separating Gr\"obner basis consists of two parts, both of which are
reduced Gr\"obner bases. However, the full set need not be the reduced $\sigma$-Gr\"obner
basis of~$I$, as the following example shows.

\begin{example}
Let $P = \mathbb{Q}[x,y,z]$, let $F = \{x-y^2\}$, let $I = \langle x-y^2,\; y^2 -z^3 \rangle$, 
and let $\sigma = \tt{Lex}$. Then $F$ is the reduced $\sigma$-Gr\"obner basis of $\langle F\rangle$ 
and $\{y^2 -z^3\}$ is the reduced $\hat{\sigma}$-Gr\"obner basis of $I\cap \mathbb{Q}[y, z]$. 
Thus $\{ x-y^2,\; y^2-z^3\}$ is a $Z$-separating {\tt Lex}-Gr\"obner basis of~$I$ 
for $Z=\{x\}$, but it is not equal to the reduced {\tt Lex}-Gr\"obner basis $\{x-z^3,\; y^2-z^3\}$ of~$I$.
\end{example}

The exact relation between a $Z$-separating Gr\"obner basis and the corresponding reduced
Gr\"obner basis of an ideal is given by the following proposition.

\begin{proposition}\label{prop:charSepGB}
In the setting of Definition~\ref{def:sepGB}, let $\sigma$ be a $Z$-separating term 
ordering for an ideal $I \subset \M$.
\begin{enumerate}
\item[(a)] The reduced $\sigma$-Gr\"obner basis of~$I$ is a $Z$-separating
$\sigma$-Gr\"obner basis of~$I$.

\item[(b)] Let $\{ \frac{1}{c_1}\, f_1,\dots, \frac{1}{c_s}\,f_s, g_1,\dots,g_t \}$
be a $Z$-separating $\sigma$-Gr\"obner basis of~$I$. For $i=1,\dots,s$, let
$h_i = \NF_{\hat{\sigma},I\cap\widehat{P}}(\tail_{z_i}(f_i))$. Then the reduced $\sigma$-Gr\"obner
basis of~$I$ is given by $H=\{ z_1 - h_1, \dots, z_s-h_s, g_1,\dots,g_t\}$.
\end{enumerate}
\end{proposition}

\begin{proof}
First we show~(a). Since~$\sigma$ is a coherently $Z$-separating term ordering, 
we know that $\{z_1,\dots,z_s\}$ is part of a minimal set of generators of $\LT_\sigma(I)$.
Hence the reduced $\sigma$-Gr\"obner basis~$G$ of~$I$ is of the form 
$G=\{ f_1,\dots,f_s, g_1,\dots,g_t\}$ with $\LC_\sigma(f_i)=1$ and $\LT_\sigma(f_i)=z_i$ 
for $i=1,\dots,s$, and with $\LT_\sigma(g_j)\in \widehat{P}$ for $j=1,\dots,t$. 
Since~$G$ is fully interreduced, we have $z_i\notin \indets(\tail_{z_i}(f_i))$
and $z_i \notin \indets (f_j)$ for $i=1,\dots,s$ and $j\ne i$, as well as
$z_i \notin \indets(g_k)$ for $k=1,\dots,t$. Therefore~$G$ is a $Z$-separated $\sigma$-Gr\"obner
basis of~$I$.

To prove~(b), it suffices to show that~$H$ is fully interreduced. Using $g_j\in \widehat{P}$
for $j=1,\dots,t$, we know that $g_j$ cannot be reduced by any element in $F=\{\frac{1}{c_1}f_1,
\dots, \frac{1}{c_s}f_s\}$. Furthermore, as~$F$ is fully interreduced, it remains to reduce
the elements of~$F$ with respect to $\widehat{G} = \{g_1,\dots,g_t\}$.
This is achieved by replacing $\tail_{z_i}(f_i)$ by 
$\NF_{\hat{\sigma},I\cap\widehat{P}}(\tail_{z_i}(f_i))$ which is its normal remainder 
under division by~$\widehat{G}$. Altogether, it follows that~$H$ is the reduced 
$\sigma$-Gr\"obner basis of~$I$.
\end{proof}

Now we are ready to prove the main result of this section. It shows that
a $Z$-separating Gr\"obner basis allows us to embed the affine scheme defined by
a polynomial ideal into a lower dimensional affine space.
Recall that, for a set of distinct indeterminates $Z = \{z_1,\dots,z_s\}$ in~$X$, 
Notation~\ref{not:YandZ} provides a disjoint union
$X = Y \cup Z$, and that we let $\widehat{P}=K[Y] = K[y_1,\dots,y_{n-s}]$.

\begin{theorem}{\bf (The $Z$-Separating Re-embedding)}\label{thm:goodIso}\\
Let $I\subseteq \M$ be an ideal, let $Z = \{z_1,\dots,z_s\}$ be a set of distinct 
indeterminates in~$X$, let~$\sigma$ be a $Z$-separating term ordering for~$I$,
and let $\{ \frac{1}{c_1}\, f_1,\dots, \frac{1}{c_s}\,f_s,\allowbreak g_1,\dots,g_t \}$
be a $Z$-separating Gr\"obner basis of~$I$, where $(f_1,\dots,f_s)$ is a coherently
$Z$-separated tuple of polynomials, where $c_i=\LT_\sigma(f_i)$ and $z_i=\LT_\sigma(f_i)$
for $i=1,\dots,s$, and where $\{g_1,\dots,g_t\}$ is the reduced $\hat{\sigma}$-Gr\"obner 
basis of $I\cap\widehat{P}$.
\begin{enumerate}
\item[(a)] Let~$\phi$ be the $K$-algebra homomorphism $\phi:\; P \longrightarrow
\widehat{P}$ given by $x_i\mapsto y_j$ if $y_j=x_i$ for some $j\in\{1,\dots,n-s\}$
and by $x_i\mapsto \tail_{z_j}(f_j)$ if $z_j=x_i$ for some $j\in\{1,\dots,s\}$.
Then~$\phi$ induces an isomorphism of $K$-algebras
$$
\Phi:\; P/I \;\longrightarrow\; \widehat{P} / (I\cap \widehat{P})
$$
which is defined by $\Phi(x_i+I) = \phi(x_i) + (I\cap \widehat{P})$ for $i=1,\dots,n$.

\item[(b)] The map $\Phi^{-1}:\; \widehat{P} / (I\cap \widehat{P}) \longrightarrow P/I$
is given by $y_i + (I\cap \widehat{P}) \mapsto y_i+I$ for $i=1,\dots,n-s$.

\item[(c)] The maps $\Phi$ and $\Phi^{-1}$ do not depend on the choice of a $Z$-separating
Gr\"obner basis of~$I$. In particular, we may use the reduced $\sigma$-Gr\"obner basis of~$I$
to define them.
\end{enumerate}
\end{theorem}

\begin{proof}
To prove~(a), we consider the composition~$\psi$ of~$\phi$ with the canonical
epimorphism $\widehat{P} \longrightarrow \widehat{P}/(I\cap \widehat{P})$ and
show that $I = \Ker(\psi)$. To prove the inclusion $\subseteq$, we note that
$\frac{1}{c_i}\, f_i = z_i - \tail_{z_i}(f_i)$, and that 
$\tail_{z_i}(f_i) \in \widehat{P}$ implies 
$$
\psi(\tfrac{1}{c_i}f_i) = 
\tail_{z_i}(f_i) - \tail_{z_i}(f_i) + (I\cap \widehat{P}) = 0
$$
for $i=1,\dots,s$. Furthermore, since $g_j\in I \cap \widehat{P}$, we also
have $\psi(g_j)=0$ for $j=1,\dots,t$.

To show the inclusion $\supseteq$, we take a polynomial $h\in \Ker(\psi)$
and note that $\tilde{h} = \NF_{\sigma,I}(h) \in \widehat{P}$.
Hence there exist $k_1,\dots,k_s,\ell_1,\dots,\ell_t \in P$ such that 
$h = k_1 f_1 + \cdots + k_s f_s + \ell_1 g_1 + \cdots + \ell_t g_t + \tilde{h}$.
Since we have shown already that $\psi(f_i)=\psi(g_j)=0$, it follows that
$\psi(\tilde{h})=0$, and hence $\tilde{h} \in I$. Altogether, we get $h\in I$, as 
claimed.

Using the equality $I=\Ker(\psi)$, we see that~$\psi$ induces an injective
$K$-algebra homomorphism $\Phi:\; P/I \;\longrightarrow\; \widehat{P} / (I\cap \widehat{P})$
which is also surjective, because $\psi(y_i)= y_i + (I\cap\widehat{P})$ for $i=1,\dots,n-s$.

Claim~(b) follows immediately from~(a). To show (c), it suffices to prove that the map~$\Phi$
defined via the set $\{ \frac{1}{c_1}\, f_1,\dots, \frac{1}{c_s}\,f_s,\allowbreak g_1,\dots,g_t \}$ 
agrees with the map $\Psi$ defined analogously via the reduced 
$\sigma$-Gr\"obner basis~$G$ of~$I$, because then an application of Proposition~\ref{prop:charSepGB}.a
finishes the proof. By part~(b) of this proposition, the set~$G$ is of the form
$G = \{ f'_1, \dots,f'_s, g_1,\dots,g_t\}$  where $f'_i = z_i - h_i$ and
$h_i = \NF_{\hat{\sigma},I\cap\widehat{P}}(\tail_{z_i}(f_i))$ for $i=1,\dots,s$.
For an indeterminate $x_i$ such that $i\in \{1,\dots,n\}$ and $x_i=z_j$ with $j\in \{1,\dots,s\}$,
it follows that 
$$
\Phi(x_i) \;=\; \tail_{z_j}(f_j) + I\cap \widehat{P} \;=\; h_j + I \cap \widehat{P} = \Psi(x_i)
$$
because the polynomial $\tail_{z_j}(f_j)$ and its normal form with respect to $I \cap \widehat{P}$
differ by an element of $I \cap \widehat{P}$. Consequently, we get $\Phi=\Psi$, as claimed.
\end{proof}

The isomorphism~$\Phi$ constructed in part~(a) of this theorem seems to depend on the
fact that we fixed a $Z$-separating term ordering~$\sigma$ for the given ideal~$I$. Our next proposition
shows that it is in fact independent of the choice of~$\sigma$.

\begin{proposition}\label{prop:equivZsepGB}
Let~$I\subseteq \M$ be an ideal in~$P$, let $Z=\{z_1,\dots,z_s\}$ be a set of
distinct indeterminates in~$\X$, and let $\sigma,\sigma'$ be $Z$-separating term
orderings for~$I$.  Then the corresponding
$K$-algebra isomorphisms $\Phi,\Phi':\; P/I \longrightarrow \widehat{P} / (I\cap\widehat{P})$
are equal.
\end{proposition}

\begin{proof}
According to Theorem~\ref{thm:goodIso}.c, we may use the reduced Gr\"obner bases of~$I$ 
with respect to~$\sigma, \sigma'$ to define~$\Phi, \Phi'$
and write them as $G  = \{f_1, \dots, f_s,\, g_1,\dots,g_t\}$ and 
$G' = \{f'_1, \dots, f'_s,\, g'_1,\dots,g'_{t'}\}$,
where $(f_1,\dots,f_s)$ and $(f'_1,\dots,f'_s)$ are coherently $Z$-se\-pa\-ra\-ted
and where $\{g_1,\dots,g_t\}$ and $\{g'_1,\dots,g'_{t'}\}$ are the reduced Gr\"obner bases of
$I\cap\widehat{P}$ with respect to~$\sigma$ and~$\sigma'$.

For every $x_i\in X\setminus Z$, we have $\Phi(x_i+I) = x_i + (I\cap \widehat{P}) = \Phi'(x_i+I)$.
Now we consider $z_i\in Z$. We have $\Phi(z_i+I) = \tail_{z_i}(f_i) + (I\cap \widehat{P})$ 
and $\Phi'(z_i+I) = \tail_{z_i}(f'_i) + (I\cap \widehat{P})$. Since 
$\tail_{z_i}(f_i) - \tail_{z_i}(f'_i) =  f_i - f'_i \in I\cap \widehat{P}$,
we obtain $\Phi(z_i+I)=\Phi'(z_i+I)$, and the proof is complete.
\end{proof}

This proposition motivates the following definition.

\begin{definition}\label{def:GrEmb}
Let $I\subseteq \M$ be an ideal, let $Z = \{z_1,\dots,z_s\}$ be a set of 
distinct indeterminates in~$X$, and let $\epsilon:\; 
\widehat{P} \longrightarrow \widehat{P} / (I\cap \widehat{P})$ be the canonical epimorphism.
\begin{enumerate}
\item[(a)] The $K$-algebra isomorphism $\Phi:\; P/I \longrightarrow \widehat{P} / (I\cap \widehat{P})$
defined in Theorem~\ref{thm:goodIso}.a is called the {\bf $Z$-separating re-em\-bed\-ding} of~$I$.

\item[(b)] The map $\Spec(\Phi^{-1}\circ\epsilon):\; \X \longrightarrow \mathbb{A}^{n-s}_K$
is called the {\bf $Z$-separating re-embedding} of~$\X$.

\end{enumerate}
\end{definition}

By Proposition~\ref{prop:equivZsepGB}, the $Z$-separating re-embedding
of~$I$ (or of~$\X$) does not depend on the choice of a $Z$-separating term ordering~$\sigma$.
This is why we did not mention the choice of~$\sigma$ in this definition.
Clearly, the $Z$-separating re-embedding of~$\X$ is obtained from the given embedding 
$\X \subseteq \mathbb{A}^n_K$ via the projection along the linear subspace defined by~$Z$.
The following example illustrates the theorem.

\begin{example}\label{ex:cohsep}
In the setting of Example~\ref{ex:cohersep}, the reduced $\sigma$-Gr\"obner basis
of $I=\langle f_1,f_2 \rangle$ is $G=\{-f_1,\, g\}$.
Here we have $Z = \{y,z\}$ and $\widehat{P} = \QQ[x]$ as well as $I\cap \widehat{P} = \langle 0\rangle$.
Hence the $Z$-separating re-embedding of~$I$ is the
map $\Phi:\; P/I \longrightarrow \QQ[x]$ defined by $x+I \mapsto x$, by
$y+I \mapsto \tail_y(-f_1)$, and by $z+I \mapsto \tail_z(g)$.
Moreover, the $Z$-separating re-embedding of $\X=\Spec(P/I)$ yields an
isomorphism $\X \cong \AA^1_K$.
\end{example}

\bigskip\bigbreak
%
%

\section{Separating Re-embeddings and the Gr\"obner Fan}
\label{Separating Re-embeddings and the Groebner Fan}

In this section we consider the problem of finding good embeddings of affine schemes into
affine spaces. As before, we let $K$ be a field, let $P=K[x_1,\dots,x_n]$, 
let $\M=\langle x_1,\dots,x_n\rangle$, let~$I$ be an ideal in~$P$ which is contained in~$\M$,
let $R=P/I$, and let $\X=\Spec(R)$. The epimorphism $P \longrightarrow P/I$ corresponds to an embedding
of~$\X$ into the affine space~$\mathbb{A}^n_K$. In this setting we call $n=\dim(P)$
the {\bf ambient dimension} of the embedding of~$\X$ or of the presentation~$P/I$.

Our goal can then be restated by saying that we look for isomorphisms of affine $K$-algebras
$\phi:\; P/I  \cong P'/I'$ such that $P'=K[y_1,\dots,y_m]$, such that~$I'$ is an ideal in~$P'$, 
and such that the ambient dimension $\dim(P')$ is as small as possible.

\begin{example}
In the setting of Example~\ref{ex:separating}, we have $P=\QQ[x,y,z]$ and $\dim(P)=3$. 
We constructed an isomorphism $P/I \cong P'/I'$, where we have $P'=\QQ[y,z]$ and
$I' = \langle y^4z - yz^4 +2z^2 - 2y\rangle$. Since $\dim(P')=2$,
we have re-embedded $\X=\Spec(P/I)$ into an affine space with a lower dimension.
\end{example}

\begin{example}
In the setting of Examples~\ref{ex:cohersep} and~\ref{ex:cohsep}, we have re-embedded
the scheme $\X=\Spec(\QQ[x,y,z] / \langle x^2 -x -y, y^2 -z\rangle)$ from ambient
dimension~3 to ambient dimension~1. Since we actually found that $\X \cong \mathbb{A}^1_K$,
the ambient dimension of the new embedding is as small as possible.
\end{example}

In view of our goal to find presentations of $R=P/I$ with minimal ambient dimensions,
we introduce the following terminology.

\begin{definition}
Let $R=P/I$ be an affine $K$-algebra as above.
\begin{enumerate}
\item[(a)] A $K$-algebra isomophism $\Psi:\; P/I \longrightarrow P'/I'$, where~$P'$
is a polynomial ring over~$K$ and~$I'$ is an ideal in~$P'$, is called 
a {\bf re-embedding} of~$I$. 

\item[(b)] A re-embedding $\Psi:\; P/I \longrightarrow P'/I'$ of~$I$ is called an 
{\bf optimal re-embedding}
if every $K$-algebra isomorphism $P/I \longrightarrow P''/I''$ with a polynomial ring~$P''$
over~$K$ and an ideal~$I''$ in~$P''$ satisfies $\dim(P'')\ge \dim(P')$.

\item[(c)] For an optimal re-embedding $\Psi:\; P/I \longrightarrow P'/I'$ of~$I$, we call
$\edim(R) = \dim(P')$ the {\bf embedding dimension} of the affine $K$-algebra~$R$.

\item[(d)] A $Z$-separating re-embedding $\Phi:\; P/I \longrightarrow
\widehat{P}/(I\cap \widehat{P})$ is called an {\bf optimal separating re-embedding}, 
if for every subset $Z' \subseteq X$ 
such that~$I$ contains a tuple of $Z'$-separating polynomials, we have $\#Z' \le \#Z$.

\item[(e)] For an optimal $Z$-separating re-embedding $\Phi:\; P/I \longrightarrow
\widehat{P} / (I\cap \widehat{P})$, we call $\sepdim(P/I) = \dim(\widehat{P})
= n - \#Z$ the {\bf separating embedding dimension} of the presentation~$P/I$ of~$R$.
\end{enumerate}
\end{definition}

At this point it is clear that,  to find optimal separating re-embedding of~$I$,
we need to find term orderings 
having the maximum number of elements in their reduced Gr\"obner bases of~$I$ 
whose leading terms are indeterminates.
To  find  such reduced Gr\"obner bases we use the notion of
the {\bf Gr\"obner fan} of the ideal~$I$ which was introduced in~\cite{MR}. 
Formally, it is a subdivision of the closed non-negative orthant~$\mathbb R^n_+$ 
consisting of a finite number of polyhedral cones, such that
the cones are in one-to-one correspondence with the leading term ideals of~$I$.
We can also associate to each polyhedral cone a marked reduced Gr\"obner basis
of~$I$ and use these Gr\"obner bases to describe the Gr\"obner fan in the following way.
The notion of marked reduced Gr\"obner basis was introduced in~\cite{RS}.

\begin{definition}\label{linGFan}
Let $I$ be a proper ideal in $P=K[x_1,\dots,x_n]$.
\begin{enumerate}
\item[(a)] Given a term ordering~$\sigma$, a {\bf marked $\sigma$-Gr\"obner basis} of~$I$
is a set of pairs $\overline{G} = \{ (\LT_\sigma(g_1),g_1) ,\dots, (\LT_\sigma(g_r),g_r)\}$, where
$G=\{g_1,\dots,g_r\}$ is a $\sigma$-Gr\"obner basis of~$I$.

\item[(b)] Given a term ordering~$\sigma$ and the marked reduced $\sigma$-Gr\"obner
basis $\overline{G} = \{ (\LT_\sigma(g_1),g_1) ,\dots, (\LT_\sigma(g_r),g_r)\}$ of~$I$, we call
$$
\LI(\overline{G}) = \{ z\in X \mid \LT_\sigma(g_i) = 
z\hbox{\quad\rm for some\quad} i\in\{1,\dots,r\} \}
$$
the {\bf set of leading indeterminates} of~$\overline{G}$.

\item[(c)] The {\bf Gr\"obner fan} of~$I$ is the set of all distinct marked
reduced Gr\"obner bases of~$I$. We denote it by $\GFan(I)$.
\end{enumerate}
\end{definition}

The following example shows that it is important to distinguish between
reduced Gr\"obner bases and marked reduced Gr\"obner bases of~$I$.

\begin{example}
In the ring $P=K[x,y]$, consider the ideal $I=\langle x+y\rangle$.
Then~$I$ has only one reduced Gr\"obner basis, namely $\{x+y\}$,
but two marked reduced Gr\"obner bases, namely
$\{(x,x+y)\}$ and $\{(y,x+y)\}$. Hence the Gr\"obner fan of~$I$ consists of
two distinct elements in this case.
\end{example}

Using the language of Gr\"obner fans, we can  rephrase some of the preceding considerations
as follows.

\begin{proposition}{\bf (Gr\"obner Fans and Optimal Separating Re-embeddings)}\label{prop:GFBest}\\
Let $I\subseteq \M$ be an ideal in~$P$, and let 
$\GFan(I) = \{\overline{G}_1, \dots, \overline{G}_k \}$.
\begin{enumerate}
\item[(a)] For $i=1,\dots,k$, we have $\sepdim(P/I) \le n - \#\LI(\overline{G}_i)$.

\item[(b)] For $i\in \{1,\dots,k\}$, we have $\sepdim(P/I) = n - \#\LI(\overline{G}_i)$
if and only if $\# \LI(\overline{G}_i)$ is maximal.

\item[(c)] For $i\in \{1,\dots,k\}$ such that $\# \LI(\overline{G}_i)$ is maximal,
the $\LI(\overline{G})$-separating re-embedding $\Phi:\; P/I \longrightarrow
\widehat{P} / (I\cap\widehat{P})$ is optimal.
\end{enumerate}
\end{proposition}

\begin{proof}
This follows immediately from Theorem~\ref{thm:goodIso}.
\end{proof}

The following example shows that
an optimal separating re-embedding need not be an optimal re-embedding of~$I$.

\begin{example}\label{ex:bestemb}
Let $P=\QQ[x,y,z]$, and let $I = \langle f_1,f_2\rangle$,  where
$$
\begin{array}{lcl}
f_1 &=& x^2  -2xy  -y^2  +x  -2z \cr
f_2 &=& x^3y \! + \tfrac{7}{6}x^2y^2 \! + \tfrac{2}{3}xy^3 \! +\tfrac{1}{6}y^4\! -\tfrac{1}{6}x^3 
\!+\tfrac{1}{3}x^2y \!+\tfrac{2}{3}xy^2\! +\tfrac{1}{3}y^3 
\!+\tfrac{1}{3}x^2z \! -\tfrac{1}{3}xy \cr
&& \!-\tfrac{1}{6}y\! -\tfrac{1}{6}z
\end{array}
$$
Then a computation using \cocoa\ shows that
$\GFan(I)$ consists of 26 different marked reduced Gr\"obner bases. 
Only two of them define optimal separating re-embeddings $\Phi:\; P/I \longrightarrow
\widehat{P}/ (I\cap \widehat{P})$.

One of these two bases is $\overline{G} =  \{ (z,g_1), (x^4,g_2)\}$ where
$$
\begin{array}{lcl}
g_1 &= & z -\tfrac{1}{2}x^2 +xy +\tfrac{1}{2}y^2 -\tfrac{1}{2}x, \cr
g_2 &= & x^4 \!+4x^3y \!+6x^2y^2 \!+4xy^3\! +y^4 \!+2x^2y \!+4xy^2 \!+2y^3 
\!-\tfrac{1}{2}x^2\! -xy \!+\tfrac{1}{2}y^2 \cr
&& \!-\tfrac{1}{2}x \!-y
\end{array}
$$
Thus we get an optimal separating re-embedding $ \Phi:\; P/I \;\longrightarrow\; \widehat{P} / 
(I\cap \widehat{P})$, where $\widehat{P}/(I\cap\widehat{P})$ has ambient dimension~2.

Next we consider the $\QQ$-algebra automorphism $\Psi:\; P \longrightarrow P$ given by 
$\Psi(x) =x -y +x^2$, $\Psi(y) = y -x^2$, and $\Psi(z) =  z +x^4 +2x^3 -2x^2y +x^2 -2xy +y^2$.
Notice that its inverse $\Psi^{-1}:\; P \longrightarrow P$ is given by 
$\Psi^{-1}(x) = x+y$, $\Psi^{-1}(y) = y +(x+y)^2$, and $\Psi^{-1}(z) = z - x^2$.
For $h_1=\Psi(g_1)$ and $h_2 = \Psi(g_2)$, we have
$$
\begin{array}{lcl}
h_1 &= & x -y -2z\cr
h_2 &= & -\tfrac{1}{6}x^5 +\tfrac{1}{6}x^4y +\tfrac{1}{3}x^4z -\tfrac{1}{3}x^4 
+\tfrac{2}{3}x^3y -\tfrac{1}{3}x^2y^2 +\tfrac{2}{3}x^3z 
-\tfrac{2}{3}x^2yz -\tfrac{1}{6}x^3 \cr
& & +\tfrac{1}{2}x^2y -\tfrac{1}{2}xy^2 +\tfrac{1}{6}y^3
+\tfrac{1}{3}x^2z -\tfrac{2}{3}xyz +\tfrac{1}{3}y^2z +\tfrac{1}{6}y^2 +\tfrac{1}{6}y -\tfrac{1}{6}z
\end{array}
$$
Consequently, the map~$\Psi$ induces a $K$-algebra isomorphism $\psi:\; P/I \longrightarrow P/I'$,
where $I'=\langle h_1,h_2\rangle$.

Next we calculate that $\GFan(I')$ comprises five marked reduced Gr\"obner bases. 
The optimal separating re-embedding $\Phi':\; P/I' \longrightarrow \widehat{P}'/ (I'\cap \widehat{P}')$
of~$I'$ is obtained using the Gr\"obner basis $\{ z -y^2 +y,\;  x -2y^2 +y\}$. 
It defines a separating re-embedding $\Phi':\; P/I' \longrightarrow \QQ[y]$.

In summary, we have a $\QQ$-algebra isomorphism $\Phi' \circ \psi:\; P/I \longrightarrow \QQ[y]$ 
and conclude that $\edim(P/I)=1$, whereas $\sepdim(P/I)= 3 - \#\LI(\overline{G}) = 2$.
\end{example}

Several marked reduced Gr\"obner bases can correspond to the same separating re-embedding,
as Proposition~\ref{prop:equivZsepGB} indicates.
Hence it is natural to define the following equivalence relation.

\begin{definition}
Let $I\subseteq\M$ be an ideal in~$P$, 
and let $\GFan(I)=\{ \overline{G}_1, \dots, \overline{G}_k\}$.
Then the equivalence relation $\sim_{\LI}$ on $\GFan(I)$ defined by
$$
\overline{G}_i \sim_{\LI} \overline{G}_j \quad\Longleftrightarrow\quad 
\LI(\overline{G}_i) = \LI(\overline{G}_j)
$$
is called the {\bf leading indeterminate set equivalence} on $\GFan(I)$.
\end{definition}

The leading indeterminate set equivalence allows us to classify separating re-embeddings
and optimal separating re-embeddings as follows.

\begin{proposition}\label{prop:equivalence}
Let $I\subseteq\M$ be an ideal in~$P$.
\begin{enumerate}
\item[(a)] There is a 1-1 correspondence between $Z$-separating re-embeddings \\
$\Phi:\; P/I \longrightarrow \widehat{P} / (I\cap \widehat{P})$ of~$I$ and the set of
equivalence classes $\GFan(I) / \!\sim_{\LI}$.

\item[(b)] The equivalence relation $\sim_{\LI}$ induces an equivalence relation
on the set $\mathcal{G}$ of all $\overline{G}\in \GFan(I)$ such that 
$\sepdim(P/I) = n - \# \LI(\overline{G})$.

\item[(c)] There is a 1-1 correspondence between optimal separating re-embeddings
$\Phi:\; P/I \longrightarrow \widehat{P} / (I\cap \widehat{P})$ of~$I$ and the set of
equivalence classes $\mathcal{G} / \sim_{\LI}$.
\end{enumerate}
\end{proposition}

\begin{proof}
To prove~(a), we note that Proposition~\ref{prop:equivZsepGB} says that equivalent marked reduced
Gr\"obner bases of~$I$ yield identical maps~$\Phi$. Conversely, if we are given
two such re-embeddings to the same ring $\widehat{P}/ (I\cap \widehat{P})$, where $\widehat{P}=K[Y]$, 
then they correspond to the same set of indeterminates $Z=X\setminus Y$ and hence to equivalent
marked reduced Gr\"obner bases.
Claims~(b) and~(c) follow immediately from~(a).
\end{proof}

The next example illustrates this proposition.

\begin{example}\label{ex:partition}
Let $P=\QQ[x,y,z]$, and let 
$I = \langle f_1, f_2, f_3, f_4, f_5, f_6\rangle$,
where we have $f_1 = x^2 -y$, $f_2 = xy -x -z$,  $f_3 = y^2 +z^2 +2x +y +2z$, $f_4 = xz +z^2 +2x +2y +2z$,
$f_5 = yz +z^2 +3x +3y +3z$, and $f_6 = z^3 +z^2 -5x -5y -5z$.

Using \cocoa, we may check that $\GFan(I)$ consists of 13 marked reduced Gr\"obner bases
and is the disjoint union of four equivalence classes modulo~$\sim_{\LI}$.
These equivalence classes are represented by the following marked reduced Gr\"obner bases: 
\begin{align*}
\overline{G}_1 \;=\; \{ &(x,\, x +\tfrac{1}{3}yz +\tfrac{1}{3}z^2 +y +z),\;
(y^2,\, y^2 -\tfrac{2}{3}yz +\tfrac{1}{3}z^2 -y),\cr
&(yz^2,\, yz^2 +\tfrac{10}{3}yz -\tfrac{2}{3}z^2),\;  
(z^3,\, z^3{+}\tfrac{5}{3}yz {+}\tfrac{8}{3}z^2) \} \cr
\overline{G}_2 \;=\; \{ &(z,\, z -xy +x),\;  (x^2,\, x^2 -y),\;  
(y^3,\, y^3 +2xy -y^2 +2y),\cr  
&(xy^2,\,xy^2 -y^2 +2y) \} \cr
\overline{G}_3 \;=\; \{ & (y,\, y +\tfrac{1}{2}xz +\tfrac{1}{2}z^2 +x +z),\;  
(x^2,\, x^2 +\tfrac{1}{2}xz +\tfrac{1}{2}z^2 +x +z),\cr
&(xz^2,\, xz^2 + \tfrac{5}{2}xz -\tfrac{1}{2}z^2),\; 
(z^3,\, z^3 +\tfrac{5}{2}xz +\tfrac{7}{2}z^2) \} \cr
\overline{G}_4 \;=\; \{ & (y,\, y -x^2),\; (z,\, z -x^3 +x),\; 
(x^5,\, x^5 -x^4 +2x^2) \} 
\end{align*}
Here the equivalence class of~$\overline{G}_4$ contains only one marked
reduced Gr\"obner basis, namely~$\overline{G}_4$.
It is the only equivalence class which defines an optimal separating re-embedding of~$I$. 
Here we have $Z = \{y,z\}$ and 
the resulting separating re-embedding is the map $\Phi:\; P/I \longrightarrow \QQ[x] / I'$
given by $x+I \mapsto x+I'$, $y+I \mapsto x^2+I'$ and $z+I \mapsto x^3-x + I'$, where
$I'= \langle x^5 -x^4 +2x^2\rangle$. In particular, we see that $\#\LI(\overline{G}_4)=2$,
and hence $\sepdim(P/I)=1$.

Notice that in the current example we even have $\edim(P/I)=1$, since the ring $P/I \cong \QQ[x]/I'$ 
has a singularity at the origin (see also Example~\ref{ex:partitioncont}).
\end{example}

\bigskip\bigbreak
%
%

\section{Cotangent Spaces and Optimal Re-embeddings}
\label{Cotangent Spaces and Optimal Re-embeddings}

In this section we investigate criteria for checking whether
an optimal separating re-embedding yields already an optimal re-embedding.
As we saw in Example~\ref{ex:bestemb}, this is not always the case.
However, we can use the size of the cotangent space at a linear maximal ideal
to bound the embedding dimension and the separating embedding dimension, as the
next theorem shows.

In the following we continue to use the setting introduced above. In particular,
let $K$ be a field, let $P=K[x_1,\dots,x_n]$, let~$I$ be an ideal in~$P$
contained in $\M=\langle x_1,\dots,x_n\rangle$, let $R=P/I$, and let $\m=\M/I$.

\begin{theorem}\label{thm:linandtang}
In the above setting, let~$\overline{G}$ be a marked reduced Gr\"obner basis of~$I$.

\begin{enumerate}
\item[(a)]  We have $\#\LI(\overline{G}) \le \dim_K(\Lin_\M(I))$.

\item[(b)] We have  
$\dim_K(\Cot_\m(P/I)) \le   \edim(P/I) \le \sepdim(P/I) \le n - \#\LI(\overline{G})$.
\end{enumerate} 
\end{theorem}

\begin{proof} 
To prove~(a), we let $\overline{G}$ be a $Z$-separating Gr\"obner basis
of~$I$ such that $\# \LI(\overline{G}) = \# Z$ is maximal.
It follows that the non-zero $\M$-linear parts $\Lin_\M(g)$ of the elements 
$g\in \overline{G}$ have the elements of~$Z$ as their leading terms. 
Hence they are $K$-linearly independent,
and Proposition~\ref{prop:linearpart} implies that they are contained in a $K$-basis of~$\Lin_\M(I)$.
This yields the claim.

To show~(b), we note that the second and third inequalities follow from the definitions
and Theorem~\ref{thm:goodIso}.
To prove the first inequality, we note that both $\edim(R)$ and the cotangent space $\Cot_\m(R) = 
\m/\m^2$ are invariants of the ring~$R=P/I$ and do not depend on the particular presentation.
Hence we may assume that $R=P/I$ is a presentation such that $\dim(P)=\edim(R)$.
Then Proposition~\ref{prop:cotancomp} yields
$\dim_K(\Cot_\m(P/I)) = n - \dim_K(\Lin_\M(I)) \le \dim(P) = \edim(R)$, as was to be shown.
\end{proof}

A particularly powerful way to use this theorem is given by the following corollary.

\begin{corollary}\label{cor:checkopt}
In the above setting, assume that $Z$ is a set of~$s$ distinct indeterminates in~$X$
and that there exists a coherently $Z$-separating tuple $(f_1,\dots,f_s)$ of
polynomials in~$I$. If $s=\dim_K(\Lin_\M(I))$ then we have $\edim(P/I)=n-s$ and
the $Z$-separating re-embedding $\Phi:\; P/I \longrightarrow \widehat{P} / (I\cap 
\widehat{P})$ is an optimal re-embedding.
\end{corollary}

\begin{proof}
The hypothesis implies $\dim_K(\Cot_\m(P/I) = n - \dim_K(\Lin_\M(I)) = n-s =
n - \#\LI(\overline{G})$. Hence we have equalities in part~(b) of the theorem, and
thus $\edim(P/I) = n-s$.
\end{proof}

The following examples provide some explicit instances for the application of this corollary.

\begin{example}\label{ex:bestemb2}
Let $P=\QQ[x,y,z,w,t]$, let $I = \langle f_1, f_2, f_3, f_4 \rangle$,  where
$f_1= x^2 +x-z +t$, $f_2 = z^2 -w^2  -w$, $f_3 = w^2 -y +w -t$, and $f_4 = x^2 +w^2 -w$,
and let $\m = \langle x, y, z, w, t\rangle /I$.

Using \cocoa, we may check that the Gr\"obner fan of~$I$ comprises 462 marked reduced Gr\"obner bases. 
The number of leading indeterminate set equivalence classes is 7, and the number of 
equivalence classes which correspond to  optimal separating re-embeddings of~$I$ is two.   
A marked reduced Gr\"obner basis which represents one of these two equivalence classes is
\begin{align*}
\overline{G} \;=\;  \{ &(t,\, t +x^2 +x -z),\;  (y,\, y -x^2 -z^2 -x +z),\;  
(w,\, w -\tfrac{1}{2}x^2 -\tfrac{1}{2}z^2),\cr  &(x^4,\, x^4 +2x^2z^2 +z^4 +2x^2 -2z^2)  \}
\end{align*}
Here we have $\#\LI(\overline{G})=3$ and $\dim(P)-\#\LI(\overline{G})=2$.

On the other hand, it is easy to check that $\Lin_\M(I)= \langle x-z+t,\, y+t,\, w\rangle_K$.
Therefore the corollary shows that $\edim(R)=2$ and
the $Z$-separating re-embedding $\Phi:\; P/I \longrightarrow \QQ[x,z] / 
\langle x^4 +2x^2z^2 +z^4 +2x^2 -2z^2\rangle$ is an optimal re-embedding.
\end{example}

A similar argument works in the setting of Example~\ref{ex:partition}.

\begin{example}\label{ex:partitioncont}
As in Example~\ref{ex:partition}, let $P=\QQ[x,y,z]$, and let 
$I = \langle f_1, \dots, f_6\rangle$,
where we have $f_1 = x^2 -y$, $f_2 = xy -x -z$,  $f_3 = y^2 +z^2 +2x +y +2z$, $f_4 = xz +z^2 +2x +2y +2z$,
$f_5 = yz +z^2 +3x +3y +3z$, and $f_6 = z^3 +z^2 -5x -5y -5z$.
We have already seen that, for $Z=\{y,z\}$, there is a $Z$-separating re-embedding 
$\Phi:\; P/I \longrightarrow \QQ[x] / \langle x^5 -x^4 +2x^2\rangle$.

Since $\Lin_\M(I) = \langle x+z,\, y\rangle_K$, we get $\dim_{\QQ}(\Lin_\M(I)) = 2$.
Using the corollary, we conclude that $\edim(P/I)=1$ and that $\Phi$ is an optimal 
re-embedding. 
\end{example}

Let us also note that any $Z$-separating re-embedding $\Phi:\; P/I \longrightarrow 
\widehat{P} / (I\cap \widehat{P})$ such that $I\cap \widehat{P}= \langle 0\rangle$
is an optimal re-embedding, since the cotangent space of the zero ideal in~$\widehat{P}$
has dimension $\dim(\widehat{P})$. The next example shows that, in general, 
the problem of detecting an optimal embedding can be difficult to deal with.

\begin{example}\label{ex:Crachiola}
Let $P=\QQ[x,y,z,t]$, and let $I=\langle f\rangle$, where $f = x +x^2y +z^2 + t^3$.  
Clearly, the only reduced Gr\"obner basis of~$I$ is $\{f\}$.
As~$f$ is not separating for any indeterminate, the optimal separating re-embedding
of~$P/I$ is the identity map.

But is it an optimal re-embedding, i.e., is $\edim(P/I)=4$? 
The answer is yes, but to prove this is not an easy task, as shown in~\cite{Cr}.
\end{example}

The final example contains an application of Corollary~\ref{cor:checkopt}
to get an optimal re-embedding of a singular border basis scheme.
In this example we use a bit of theory of border bases schemes
(see for instance~\cite{KR2}, \cite{KR4}, \cite{KR5}, \cite{KSL}, and \cite{Rob}).
As said in the introduction, we will devote another paper to an extensive study of 
good, possibly optimal, embeddings of border bases schemes.

\begin{example}\label{ex:singBBS}
Let $P=\QQ[x,y,z]$, and let $\OO= \{1, z, y, x\}$. 
Then the border of~$\OO$ is given by
$\partial\OO=\{ z^2,\, yz,\, xz,\, y^2,\,  xy,\,  x^2\}$. 
This yields $\mu=4$ and $\nu=6$. Consequently, the polynomial ring $\QQ[C]$
has 24 indeterminates.

The computation of the linear part of~$I(\BO)$ using Proposition~\ref{prop:cotancomp}
shows that we have $\Lin_\M(I(\BO)) = \langle c_{11},\dots,c_{16}\rangle_{\QQ}$, and therefore 
$n\,\mu = 12$ implies that the monomial point $p=(0,\dots,0)$ is a singular point of~$\BO$
(see also~\cite{KSL}, Example 4.3).
In particular, it follows that~$\BO$ is not isomorphic to an affine space.
Nevertheless, let us try to embed~$\BO$ optimally. There exists a set of coherently $Z$-separating
polynomials $\{f_1,\dots,f_6\}$ in~$I(\BO)$, where $Z=\{c_{11},\dots,c_{16}\}$ and
\begin{align*}
f_1 \;=\; & c_{11} -c_{23}c_{41} -c_{33}c_{42}+c_{21}c_{43} -c_{43}^2 +c_{31}c_{45} +c_{41}c_{46},\cr
f_2 \;=\; & c_{12} -c_{23}c_{42} +c_{22}c_{43} -c_{33}c_{44} +c_{32}c_{45} -c_{43}c_{45} +c_{42}c_{46},\cr
f_3 \;=\; & c_{13} -c_{26}c_{41} -c_{36}c_{42} +c_{23}c_{43} +c_{33}c_{45},\cr
f_4 \;=\; & c_{14} -c_{25}c_{42} +c_{24}c_{43} -c_{35}c_{44} +c_{34}c_{45} -c_{45}^2 +c_{44}c_{46},\cr
f_5 \;=\; & c_{15} -c_{26}c_{42} +c_{25}c_{43} -c_{36}c_{44} +c_{35}c_{45},\cr
f_6 \;=\; & c_{16} +c_{26}c_{32} -c_{25}c_{33} -c_{35}^2 +c_{34}c_{36} -c_{36}c_{45} +c_{35}c_{46}
\end{align*}
Consequently, the $Z$-separating re-embedding of the ideal~$I(\BO)$
is an isomorphism $\Phi:\, B_\OO \rightarrow \widehat{P} / (I(\BO)\cap \widehat{P})$,
where $\widehat{P} = \QQ[c_{21},\dots,c_{26},c_{31},\dots,c_{36},c_{41},\dots,c_{46}]$
is a polynomial ring in~18 indeterminates. Here it turns out
that $I(\BO) \cap \widehat{P}$ is minimally generated by the 15 quadratic polynomials
$$
\scriptstyle{
\begin{array}{lll}
 h_1 &\!=\!&  c_{25}c_{32} {-}c_{24}c_{33} {+}c_{26}c_{42} {-}c_{25}c_{43} \cr
 h_2 &\!=\!&   c_{25}c_{41} {-}c_{23}c_{42} {+}c_{35}c_{42} {-}c_{33}c_{44} \cr
 h_3 &\!=\!& c_{25}c_{31} {-}c_{22}c_{33} {-}c_{36}c_{42} {+}c_{33}c_{45} \cr
 h_4 &\!=\!&  c_{26}c_{31} {-}c_{23}c_{33} {-}c_{33}c_{35} {+}c_{32}c_{36} {-}c_{36}c_{43} {+}c_{33}c_{46} \cr
 h_5 &\!=\!&    c_{22}c_{41} {-}c_{21}c_{42} {+}c_{32}c_{42} {+}c_{42}c_{43} {-}c_{31}c_{44} {-}c_{41}c_{45} \cr
 h_6 &\!=\!&  c_{23}c_{24} {-}c_{22}c_{25} {+}c_{25}c_{34} {-}c_{24}c_{35} {+}c_{26}c_{44} {-}c_{25}c_{45} \cr
 h_7 &\!=\!&    c_{23}c_{31} {-}c_{21}c_{33} {+}c_{32}c_{33} {-}c_{31}c_{35} {-}c_{36}c_{41} {+}c_{33}c_{43} \cr
 h_8 &\!=\!&     c_{24}c_{41} {-}c_{22}c_{42} {+}c_{34}c_{42} {-}c_{32}c_{44} {+}c_{43}c_{44} {-}c_{42}c_{45} \cr
 h_9 &\!=\!&    c_{23}c_{25} {-}c_{22}c_{26} {+}c_{25}c_{35} {-}c_{24}c_{36} {+}c_{26}c_{45} {-}c_{25}c_{46} \cr
 h_{10} &\!=\!&   c_{23}c_{32} {-}c_{22}c_{33} {+}c_{33}c_{34} {-}c_{32}c_{35} {+}c_{26}c_{41} {-}c_{23}c_{43}
                 {+}c_{35}c_{43} {-}c_{33}c_{45} \cr
 h_{11} &\!=\!&   c_{24}c_{31} {-}c_{22}c_{32} {-}c_{23}c_{42} {-}c_{35}c_{42} {+}c_{22}c_{43} {+}c_{32}c_{45}
                  {-}c_{43}c_{45} {+}c_{42}c_{46} \cr
 h_{12} &\!=\!&   c_{22}c_{23} {-}c_{21}c_{25} {+}c_{24}c_{33} {-}c_{22}c_{35} {-}c_{26}c_{42} {+}2c_{25}c_{43}
                  {-}c_{36}c_{44} {-}c_{23}c_{45}  {+}c_{35}c_{45} \cr
 h_{13} &\!=\!&   c_{22}^2 {-}c_{21}c_{24} {+}c_{24}c_{32} {-}c_{22}c_{34} {+}c_{24}c_{43} {-}c_{23}c_{44} 
                  {-}c_{35}c_{44} {+}c_{34}c_{45} {-}c_{45}^2 {+}c_{44}c_{46} \cr
 h_{14} &\!=\!&   c_{22}c_{31} {-}c_{21}c_{32} {+}c_{32}^2 {-}c_{31}c_{34} {-}c_{23}c_{41} {-}c_{35}c_{41} 
                  {+}c_{21}c_{43} {-}c_{43}^2 {+}c_{31}c_{45} {+}c_{41}c_{46} \cr
 h_{15} &\!=\!&   c_{23}^2 {-}c_{21}c_{26} {+}c_{26}c_{32} {-}c_{35}^2 {-}c_{22}c_{36} {+}c_{34}c_{36} 
                  {+}c_{26}c_{43} {-}c_{36}c_{45} {-}c_{23}c_{46} {+}c_{35}c_{46} 
\end{array}
}
$$
Altogether, since $\dim_\QQ(\Lin_\M(I(\BO))) = 6 = \#Z$, 
we apply Corollary~\ref{cor:checkopt} and conclude that $\edim(B_\OO)=18$, that~$\Phi$ 
yields an optimal embedding of the 12-dimensional scheme $\BO$ into $\AA^{18}_{\QQ}$, 
and that the vanishing ideal of the image is minimally generated by~15 quadrics.
\end{example}

%
%

\medskip
\section*{Acknowledgements}
The second and third authors thank the University of Passau for its hospitality
during part of the preparation of this paper. The first and second authors were 
partially supported by the Vietnam National Foundation for Science and Technology Development
(NAFOSTED) grant number 101.04-2019.07.

%
%

\bigbreak

\end{document}